\documentclass[preprint,authoryear,12pt]{elsarticle}

\usepackage{amssymb}
\usepackage{amsfonts}
\usepackage{amsmath}

\makeatletter
\def\ps@pprintTitle{
 \let\@oddhead\@empty
 \let\@evenhead\@empty
 \def\@oddfoot{\centerline{\thepage}}
 \let\@evenfoot\@oddfoot}
\makeatother

\newtheorem{theorem}{Theorem}

\newtheorem{definition}{Definition}

\newtheorem{lemma}{Lemma}

\newtheorem{remark}{Remark}

\newenvironment{proof}[1][Proof]{\noindent\textbf{#1.} }{\ \rule{0.5em}{0.5em}}
\begin{document}

\begin{frontmatter}%
\title{Estimation of Stable Distribution Parameters from a Dependent Sample}%
\author{A. W. Barker}%
\address{Department of Statistics, Macquarie University, Sydney, NSW 2109 Australia }%

\begin{abstract}
Existing methods for the estimation of stable distribution parameters, such as those based on sample quantiles, sample characteristic functions or
maximum likelihood generally assume an independent sample. Little attention has been paid to estimation from a dependent sample.
In this paper, a method for the estimation of stable distribution parameters from a dependent sample is proposed based on the sample quantiles.
The estimates are shown to be asymptotically normal. The asymptotic variance is calculated for stable moving average processes.
Simulations from stable moving average (\textsc{sma}) processes are used to demonstrate these estimators.
\end{abstract}
\begin{keyword}%
Quantile \sep Stable\ Distribution \sep Moving Average

\end{keyword}%
\end{frontmatter}%

\section{Introduction}
\label{sec:int}
A number of methods have been proposed for the estimation of the parameters of a stable distribution have been proposed. 
A method based on sample quantiles was proposed by \cite{ci:FR71} which was simple to implement, but was only
applicable to symmetric stable distributions with $\alpha \geq 1$ and contained a slight bias. This method was extended by
\cite{ci:M86} to cover asymmetric stable distributions and which is asymptotically unbiased. Other methods have been proposed
based on the sample characteristic function, (\cite{ci:P72}, \cite{ci:PHL75} and \cite{ci:KW98}).  Maximum likelihood estimation
methods have been proposed by \cite{ci:BY90} and \cite{ci:N01}. For a discussion on the use of indirect inference for the estimation
of stable distributions, see \cite{ci:GRV11}.

The sample quantile method of \cite{ci:M86} assumes an independent sample. The primary goal of this paper is to investigate the
extension of this method to cover dependent samples. For that purpose we use results on quantile estimation from dependent samples
which show under certain conditions, these estimates are consistent and asymptotically normal (e.g. \cite{ci:S68} and \cite{ci:DHOV13}).
We conclude with some simulations.

Throughout this paper we use the $\textsc{sma}$(q) process as an example of a dependent stable process. An $\textsc{sma}$(q) process 
$\{X_t\}$ is defined as follows
\begin{equation}
X_t = \sum_{j=0}^{q}\theta_j \varepsilon_{t-j}
\label{eq:int:1}
\end{equation}
where $\theta_0=1$ and $\{\varepsilon_t\}$ is an independent identically distributed (iid) sequence of stable random variables such that
\begin{equation}
\varepsilon_t \sim S_{\alpha }^{0}\left( \beta_0 ,\gamma_0 ,\delta_0 \right)
\label{eq:int:2}
\end{equation}
using the $S^0$ parameterisation of stable random variables in \cite{ci:N98}. Using the properties of the $S^0$ parameterisation
given in Lemma 1, \cite{ci:N98} it can be shown that  $X_t$ also has a stable distribution,
\begin{equation}
X_t \sim S_{\alpha }^{0}\left( \beta ,\gamma ,\delta \right).
\label{eq:int:3}
\end{equation}
 Formulae for the stable distribution parameters of $X_t$ in terms of the $\textsc{sma}$(q) process parameters can be found in \cite{ci:B14}.

\section{Quantile Estimation from a Dependent Sample}
\label{sec:qe}
For any real-valued random variable $X$ on a probability space $\left(
\Omega ,\mathcal{F},P\right) ,$ there is an associated distribution function 
$F:\mathbb{R} \rightarrow \left[ 0,1\right] $ defined by
\begin{equation}
F\left( x\right) \equiv P\left( X\leq x\right) .
\label{eq:qe:1a}
\end{equation}
The $p$th quantile, $\xi_{p}$, of $F$ is defined by
\begin{equation}
\xi_{p} \equiv \inf \left\{ x:F\left( x\right) \leq p\right\}.
\label{eq:qe:1b}
\end{equation}
The density function, $f:\mathbb{R} \rightarrow \mathbb{R}_{+}$, of $F$ is defined by
\begin{equation}
F\left( x\right) \equiv \int_{-\infty }^{x}f\left( s\right) ds.
\label{eq:qe:2}
\end{equation}

Let $\left\{ x_{j}\right\} _{j=1}^{n}$ be a sample drawn from random variables with the distribution function $F$. From this sample we define the empirical distribution function
and empirical quantile estimators by 
\begin{equation}
\widehat{F}_{n}\left( x\right) =\dfrac{1}{n}\sum_{j=1}^{n}I_{\left( -\infty,x\right] }\left( x_{j}\right),
\label{eq:qe:3}
\end{equation}
and
\begin{equation}
\widehat{\xi }_{p}=\inf \left\{ x:\widehat{F}_{n}\left( x\right) \geq p\right\}.
\label{eq:qe:4}
\end{equation}
There is an extensive literature about the statistical properties of the empirical estimators
(e.g. \cite{ci:C46} and \cite{ci:S80}).

The following theorems assume that $\left\{ x_{j}\right\} _{j=1}^{n}$ is an
iid sample. The first theorem shows that the
empirical quantile estimator has strong consistency wherever the underlying
distribution function is not flat.

\begin{theorem}
\label{th:qe:1}\textbf{(Strong Consistency of }$\widehat{\xi }_{p}$ - 
\textbf{\cite{ci:S80}, Theorem 2.3.1).} Let $0<p<1.$ If $\xi _{p}$ is the
unique solution $x$ of $F\left( x_{-}\right) \leq p\leq F\left( x\right) ,$
then $\widehat{\xi }_{p}$ is a strongly consistent estimator of $\xi _{p}.$
\end{theorem}

The next theorem shows that the empirical quantile estimator is
asymptotically normal under some conditions on the underlying distribution
function (See also \cite{ci:C46})

\begin{theorem}
\label{th:qe:2}\textbf{(Asymptotic Normality of Empirical Quantile
Estimator - \cite{ci:S80}, Corollary 2.3.3B).} For $0<p<1,$ if $F$ possesses
a density $f$ in a neighbourhood of $\xi _{p}$ and if $f$ is positive and
continuous at $\xi _{p},$ then 
\begin{equation}
\widehat{\xi }_{p}\text{ is }AN\left( \xi _{p},\dfrac{p\left( 1-p\right) }{f^{2}\left( \xi _{p}\right) n}\right).
\label{eq:qe:5}
\end{equation}
\end{theorem}

Theorem \ref{th:qe:2} can be extended to cover the estimation of multiple quantiles from a single sample.

\begin{theorem}
\label{th:qe:3}\textbf{(Asymptotic Covariances of Empirical Quantile
Estimators - \cite{ci:S80}, Theorem 2.3.3B).} Let $0<p_{1}<\cdots <p_{k}<1.$
Suppose that $F$ has a density $f$ in a neighbourhoods of $\xi
_{p_{1}},\ldots ,\xi _{p_{k}}$ and that $f$ is positive and continuous at $
\xi _{p_{1}},\ldots ,\xi _{p_{k}}.$ Let $\widehat{\xi}=\left( \widehat{
\xi }_{p_{1}},\ldots ,\widehat{\xi }_{p_{k}}\right) ^{\prime }$ denote
the empirical quantiles estimates of $\xi =\left( \xi _{p_{1}},\ldots ,\xi
_{p_{k}}\right) ^{\prime },$ then
\begin{equation}
\sqrt{n}\left( \widehat{\xi }-\xi \right) \overset{d}{\longrightarrow } N\left( 0,\Sigma \right).
\label{eq:qe:7}
\end{equation}
The element in the $ith$ row and $jth$ column of $\Sigma $ is given by 
\begin{equation}
\sigma _{ij}=\left\{ 
\begin{array}{cc}
\dfrac{p_{i}\left( 1-p_{j}\right) }{f\left( \xi _{p_{i}}\right) f\left( \xi
_{p_{j}}\right) } & \text{ for }i\leq j \\ 
\dfrac{p_{j}\left( 1-p_{i}\right) }{f\left( \xi _{p_{i}}\right) f\left( \xi
_{p_{j}}\right) } & \text{ for }i>j%
\end{array}
\right.
\label{eq:qe:8}
\end{equation}
\end{theorem}

The asymptotic distributions listed in Theorems \ref{th:qe:2} and \ref{th:qe:3} only apply if the sample $\left\{ x_{j}\right\} _{j=1}^{n}$ is
iid. The asymptotic distribution of the empirical quantile estimator, where the sample is taken from a possibly non-stationary
m-dependent process was derived by \cite{ci:S68}. Further work in this area has been done by, amongst others: \cite
{ci:DS71} on autoregressive processes, \cite{ci:S72} on $\phi $ - mixing processes, \cite{ci:OH05} on non-stationary processes and \cite{ci:DHOV13}
on S-mixing processes. In this paper, we use the results of \cite{ci:DHOV13} for S-mixing processes.

\begin{definition}
\label{def:qe:1}
\textbf{(S-mixing Process - \cite{ci:BHS09}).} A process $\left\{ \mathbf{X}_{t}\right\} $ is called S-mixing if it satisfies the following conditions
\begin{enumerate}
\item For any $t\in \mathbb{Z}$ and $m\in \mathbb{N},$ one can find a random variable $\mathbf{X}_{tm}$ such that 
\begin{equation}
P\left( \left\vert \mathbf{X}_{t}-\mathbf{X}_{tm}\right\vert \geq \gamma_{m}\right) \leq \delta _{m}
\label{eq:qe:9}
\end{equation}
for some numerical sequences $\gamma _{m}\rightarrow 0,\delta_{m}\rightarrow 0.$
\item For any disjoint intervals $I_{1},\ldots ,I_{r}$ of integers and any positive integers $m_{1},\ldots ,m_{r},$ 
the vectors $\left\{ \mathbf{X}_{tm_{1}},t\in I_{1}\right\} ,\ldots ,\left\{ \mathbf{X}_{tm_{r}},t \in I_{r}\right\}$
are independent provided the separation between $I_{k}$ and $I_{l}$ is greater than $m_{k}+m_{l}.$
\end{enumerate}
\end{definition}

An \textsc{sma} process is an S-mixing process (\cite{ci:BHS09}) and also a $\phi$ - mixing process (e.g. \cite{ci:D94}). Other examples of S-mixing processes
can be found in \cite{ci:BHS09} and \cite{ci:DHOV13}.

Let 
\begin{equation}
\mathbf{X}_{t}=g\left( e_{t},e_{t-1},\ldots \right) ,\quad t\in \mathbb{Z}
\label{eq:qe:10}
\end{equation}
be a $k$ - dimensional process where $\left\{ e_{t}\right\} $ is an iid sequence of elements from the measurable space $\Omega $
and $g:\Omega^{\infty }\rightarrow \mathbb{R}^{k}$ is a measurable function. The following theorem from \cite{ci:DHOV13} provides the asymptotic distribution for the
empirical quantile estimators from a multivariate S-mixing process. Note that Theorem 6.5 in \cite{ci:S72} proves a similar result for $\phi $ - mixing processes

\begin{theorem}
\label{th:qe:4}
\textbf{(\cite{ci:DHOV13}, Theorem 1).} Let $\left\{\mathbf{X}_{t}\right\} $ be a stationary process satisfying 
$\left( \ref{eq:qe:10}\right) $ and let $\mathbf{\xi }=\left( \xi _{1},\ldots ,\xi_{k}\right) ^{\prime }$ denote the quantiles of
$\left\{ \mathbf{X}_{t}\right\} $ at $\mathbf{p}=\left( p_{1},\ldots ,p_{k}\right) ^{\prime }.$
Suppose that
\begin{description}
\item{A1} For each $i$ in $1,\ldots ,k$ the marginal distribution function $F_{i}\left( x\right) $ has a density $f_{i}\left( x\right) $ that is
positive and continuous in the neighbourhood of $\xi _{i}$ and $f_{i}\left(x\right) $ is uniformly bounded by some constant $B.$
\item{A2} The process $\left\{ \mathbf{X}_{t}\right\} $ is S-mixing with coefficients $\gamma _{m}=\delta _{m}=O\left( m^{-A}\right) $ where $A>4.$
\end{description}
Then 
\begin{equation}
\sqrt{n}\left( \widehat{\mathbf{\xi }}-\mathbf{\xi }\right) \overset{d}{\longrightarrow }N\left( 0,\Sigma \right)
\label{eq:qe:11}
\end{equation}
where
\begin{eqnarray}
\Sigma &=&V^{-1}QV^{-1},
\label{eq:qe:12} \\
V &=&diag\left( f_{1}\left( \xi _{1}\right) ,\ldots ,f_{k}\left(\xi _{k}\right) \right),
\label{eq:qe:13} \\
Q &=&\sum_{h\in \mathbb{Z}}E\left[ Q_{0}Q_{h}^{\prime }\right],
\label{eq:qe:14} \\
Q_{j} &=&\left( I\left\{ X_{j;1}\leq \xi _{1}\right\} -p_{1},\ldots,I\left\{ X_{j;k}\leq \xi _{k}\right\} -p_{k}\right) .
\label{eq:qe:15}
\end{eqnarray}
The element is the $ith$ row and $jth$ column of the matrix $\Sigma $ is given by
\begin{equation}
\sigma _{ij}=\dfrac{\sum_{h\in \mathbb{Z}}\left( P\left( \left\{ X_{t;i}\leq \xi _{i}\right\} \cap \left\{X_{t+h;j}\leq \xi _{j}\right\} \right) -p_{i}p_{j}\right) }
{f_{i}\left( \xi_{i}\right) f_{j}\left( \xi _{j}\right) }.
\label{eq:qe:16}
\end{equation}
\end{theorem}

\begin{remark}
\label{rem:qe:1}Note that all stable distributions satisfy Assumption A1 and that all \textsc{arma} processes satisfy Assumption A2
\end{remark}

\begin{remark}
\label{rem:qe:2}Whilst Theorem \ref{th:qe:4} applies to the
estimation of a single quantile from each component of a vector process, it
can be adapted for the joint estimation of multiple quantiles from a scalar
process $\left\{ X_{t}\right\} $ through application to the vector process 
\begin{equation}
\left\{ \mathbf{X}_{t}\right\} =\left( X_{t},\ldots ,X_{t}\right) ^{\prime }
\label{eq:qe:17}
\end{equation}
\end{remark}

In order to calculate the asymptotic variance, $\Sigma$, of the empirical
quantile estimates from a scalar S-mixing process $\left\{ X_{t}\right\} $,
it is necessary to calculate the joint probabilities 
\begin{equation}
G_{h}\left( \xi _{i},\xi _{j}\right) =P\left( \left\{ X_{t}\leq \xi
_{i}\right\} \cap \left\{ X_{t+h}\leq \xi _{j}\right\} \right)
\label{eq:qe:18}
\end{equation}%
for each $h\in \mathbb{Z}$. For $h=0,$ $\left( \ref{eq:qe:18}\right) $
can be simplified to give%
\begin{equation}
G_{0}\left( \xi _{i},\xi _{j}\right) =\min \left( p_{i},p_{j}\right)
\label{eq:qe:19}
\end{equation}%
For $h\neq 0,$ the evaluation of $\left( \ref{eq:qe:18}\right) $ whilst
theoretically possible for some S-mixing processes is computationally very
difficult for many. For an \textsc{sma}(q) process, the independence of $%
X_{t}$ and $X_{t+h}$ for all $\left\vert h\right\vert >q$ means that $\left( %
\ref{eq:qe:16}\right) $ can be simplified to 
\begin{equation}
\sigma_{ij}=\dfrac{\sum_{h=-q}^{q}\left( G_{h}\left( \xi _{i},\xi _{j}\right)
-p_{i}p_{j}\right) }{f\left( \xi _{i}\right) f\left( \xi _{j}\right) }.
\label{eq:qe:20}
\end{equation}%
For an iid process we get%
\begin{equation}
G_{h}\left( \xi _{i},\xi _{j}\right) =0,\quad \text{for }h\neq 0.
\label{eq:qe:21}
\end{equation}%
Thus for iid processes, Theorem \ref{th:qe:3} produces the same
asymptotic covariance matrix as Theorem \ref{th:qe:4}.

Suppose $\{X_{t}\}$ is an \textsc{sma}(1) process and let $f$ and $F$ denote the density and distribution functions respectively
of the associated innovation sequence $\left\{ \varepsilon _{t}\right\} $. Then 
\begin{eqnarray}
G_{1}\left( \xi _{i},\xi _{j}\right) &=&P\left( \left\{ \varepsilon
_{t}+\theta _{1}\varepsilon _{t-1}\leq \xi _{i}\right\} \cap \left\{
\varepsilon _{t+1}+\theta _{1}\varepsilon _{t}\leq \xi _{j}\right\} \right) 
\notag \\
&=&P\left( \left\{ \varepsilon _{t-1}\leq \dfrac{\xi _{i}-\varepsilon _{t}}{%
\theta _{1}}\right\} \cap \left\{ \varepsilon _{t+1}\leq \xi _{j}-\theta
_{1}\varepsilon _{t}\right\} \right)  \notag \\
&=&\int_{-\infty }^{\infty }F\left( \dfrac{\xi _{i}-u}{\theta _{1}}\right)
F\left( \xi _{j}-\theta _{1}u\right) f\left( u\right) du,
\label{eq:qe:23}
\end{eqnarray}%
which can be evaluated numerically. Note that 
\begin{equation}
G_{1}\left( \xi _{i},\xi _{j}\right) =G_{-1}\left( \xi _{j},\xi _{i}\right)
\label{eq:qe:24}
\end{equation}

For higher order \textsc{sma}(q) processes, the evaluation of $G_{h}\left(
\xi _{i},\xi _{j}\right) $ becomes computationally difficult, involving a $%
q-1+h$ dimensional integral. However the estimation of $G_{h}\left( \xi
_{i},\xi _{j}\right) $ is straightforward. Let $\left\{ x_{t}\right\}
_{t=1}^{n}$ be a sample of size $n$ from the \textsc{sma}%
(q)\ process $\left\{ X_{t}\right\} .$ We define the estimator $\widehat{G}%
_{h}\left( \xi _{i},\xi _{j}\right) $ as%
\begin{equation}
\widehat{G}_{h}\left( \xi _{i},\xi _{j}\right) =\left( n-h\right)
^{-1}\sum_{t=1}I\left\{ x_{t}\leq \xi _{i}\right\} \cdot I\left\{
x_{t+h}\leq \xi _{j}\right\} ,\quad \text{for }\left\vert h\right\vert >1
\label{eq:qe:25}
\end{equation}%
and it is clear that $\widehat{G}_{h}\left( \xi _{i},\xi _{j}\right) $ is a
consistent estimator of $G_{h}\left( \xi _{i},\xi _{j}\right) .$ For the
purposes of this paper we do not consider the asymptotic properties of $%
\widehat{G}_{h}\left( \xi _{i},\xi _{j}\right) .$

\section{Estimation of Stable Distribution Parameters}
\label{sec:est}
The following method for the estimation of stable distribution parameters was proposed in \cite{ci:M86}. Let $\xi _{p}$ denote the $pth$ quantile of the stable 
distribution $S_{\alpha}^{0}\left(\beta ,\gamma ,\delta \right) $ and define the following statistics
\begin{eqnarray}
v_{\alpha } &=&\dfrac{\xi _{0.95}-\xi _{0.05}}{\xi _{0.75}-\xi _{0.25}},
\label{eq:est:1a} \\
v_{\beta } &=&\dfrac{\xi _{0.95}+\xi _{0.05}-2\xi _{0.50}}{\xi _{0.95}-\xi_{0.05}}.
\label{eq:est:1b}
\end{eqnarray}
These statistics do not depend on $\gamma ,\delta$  and so we can consider them as functions solely of $\alpha ,\beta,$
\begin{eqnarray}
v_{\alpha } &=&\phi _{1}\left( \alpha ,\beta \right),
\label{eq:est:2a} \\
v_{\beta } &=&\phi _{2}\left( \alpha ,\beta \right).
\label{eq:est:2b}
\end{eqnarray}

It can be seen that $\phi _{1}\left( \alpha ,\beta \right) $ is a strictly decreasing function of $\alpha $ for each $\beta $ and that 
$\phi _{2}\left(\alpha ,\beta \right) $ is a strictly decreasing function of $\beta $ for each $\alpha .$ The relationships $\left( \ref{eq:est:2a}\right) $
and $\left( \ref{eq:est:2b}\right) $ can be inverted to give
\begin{eqnarray}
\alpha &=&\psi _{1}\left( v_{\alpha },v_{\beta }\right) ,
\label{eq:est:3a} \\
\beta &=&\psi _{2}\left( v_{\alpha },v_{\beta }\right) .
\label{eq:est:3b}
\end{eqnarray}

Let $\widehat{\xi }_{p}$ denote a consistent estimator for $\xi _{p}$.
Substituting the estimators $\widehat{\xi }_{p}$ into $\left( \ref{eq:est:1a}\right) $ and
$\left( \ref{eq:est:1b}\right) $ gives consistent estimators for $v_{\alpha},v_{\beta },$
\begin{eqnarray}
\widehat{v}_{\alpha } &=&\dfrac{\widehat{\xi }_{0.95}-\widehat{\xi }_{0.05}}{\widehat{\xi }_{0.75}-\widehat{\xi }_{0.25}},
\label{eq:est:4a} \\
\widehat{v}_{\beta } &=&\dfrac{\widehat{\xi }_{0.95}+\widehat{\xi }_{0.05}-2\widehat{\xi }_{0.50}}{\widehat{\xi }_{0.95}-\widehat{\xi }_{0.05}}.
\label{eq:est:4b}
\end{eqnarray}
Consistent estimators for the parameters $\alpha ,\beta $ can then be calculated using
\begin{eqnarray}
\widehat{\alpha } &=&\psi _{1}\left( \widehat{v}_{\alpha },\widehat{v}_{\beta }\right),
\label{eq:est:5a} \\
\widehat{\beta } &=&\psi _{2}\left( \widehat{v}_{\alpha },\widehat{v}_{\beta}\right).
\label{eq:est:5b}
\end{eqnarray}

Under the $S^{0}$ parameterisation of the stable distribution, the parameters $\gamma$ and $\delta$ act respectively as scale and location parameters 
of the distribution. We formalise this property in the following lemma.
\begin{lemma}
Let
\begin{equation}
X \sim S_{\alpha}^{0}\left(\beta,\gamma,\delta\right)
\label{eq:est:6}
\end{equation}
and
\begin{equation}
X^{*} \sim S_{\alpha}^{0}\left(\beta,1,0\right)
\label{eq:est:7}
\end{equation}
be stable random variables. Let $\xi_{p}$ and $\xi_{p}^{*}$ denote respectively the $p$th quantile of $X$ and $X^{*}$.
Then for any $0 < p_{1},p_{2} < 1$ where $p_{1} \ne p_{2}$ we have
\begin{equation}
\gamma = \dfrac{\xi_{p_{2}} - \xi_{p_{1}}}{\xi_{p_{2}}^{*} - \xi_{p_{1}}^{*}}
\label{eq:est:8}
\end{equation}
and
\begin{equation}
\delta = \xi_{p_{1}} - \gamma \xi_{p_{1}}^{*}
\label{eq:est:9}
\end{equation}
\label{lem:est:1}
\end{lemma}
\begin{proof}
From Lemma 1 (\cite{ci:N98}), we have
\begin{equation}
\gamma X^{*} + \delta \sim X.
\label{eq:est:10}
\end{equation}
It follows that for any $0< p < 1$
\begin{equation}
\xi_{p}^{*} = \dfrac{\xi_{p} - \delta}{\gamma},
\label{eq:est:10}
\end{equation}
from which (\ref{eq:est:8}) and (\ref{eq:est:9}) follow immediately.
\end{proof}

We can use the results of Lemma \ref{lem:est:1} to define the estimators of $\gamma$ and $\delta$ by
\begin{equation}
\widehat{\gamma} = \dfrac{\widehat{\xi}_{0.75} - \widehat{\xi}_{0.25}}
{\widehat{\xi}_{0.75}^{*} - \widehat{\xi}_{0.25}^{*}}
\label{eq:est:11}
\end{equation}
and
\begin{equation}
\widehat{\delta} = \widehat{\xi}_{0.50} - \widehat{\gamma}\widehat{\xi}_{0.50}^{*}.
\label{eq:est:12}
\end{equation}
where $\widehat{\xi}_{p}^{*}$ is the $pth$ quantile of the distribution $S_{\widehat\alpha}^{0}(\widehat\beta,1,0)$.
The estimators in (\ref{eq:est:11}) and (\ref{eq:est:12}) are similar to those defined in \cite{ci:M86}. The differences are due to McCulloch's choice of
parameterisation for the stable distribution, which includes discontinuities at $\alpha = 1$.

From Lemma \ref{lem:est:1}, it can be seen that other choices of quantile levels are available to define $\widehat{\gamma}$ and $\widehat{\delta}$.
For computational efficiency, it is preferable to choose from the same quantile levels used to define $\widehat{\alpha}$ and $\widehat{\beta}$. Indeed, other choices 
of quantile levels are also available to define $\widehat{\alpha}$ and $\widehat{\beta}$ and it is possible that a different choice of quantile levels would
produce better estimators.

Let 
\begin{equation}
\widehat{\xi}_M=\left( \widehat{\xi }_{0.05},\widehat{\xi }_{0.25},\widehat{\xi }_{0.50},\widehat{\xi }_{0.75},\widehat{\xi }_{0.95}\right) ^{\prime }
\label{eq:est:13}
\end{equation}
denote the empirical quantile estimates of 
\begin{equation}
\xi_M=\left( \xi_{0.05},\xi_{0.25},\xi_{0.50},\xi_{0.75},\xi_{0.95}\right) ^{\prime }
\label{eq:est:14}
\end{equation}
from an S-mixing process $\{X_t\}$.  
Let $\Sigma_M$ denote the asymptotic covariance matrix of $\widehat{\xi}_M$ obtained from Theorem \ref{th:qe:4}.
We define the matrix of partial derivatives $D$ by 
\begin{equation}
D=\left(
\begin{array}{ccccc}
\dfrac{\partial \widehat{\alpha }}{\partial \widehat{\xi }_{0.05}} & 
\dfrac{\partial \widehat{\alpha }}{\partial \widehat{\xi }_{0.25}} & 
\dfrac{\partial \widehat{\alpha }}{\partial \widehat{\xi }_{0.50}} &
\dfrac{\partial \widehat{\alpha }}{\partial \widehat{\xi }_{0.75}} &
\dfrac{\partial \widehat{\alpha }}{\partial \widehat{\xi }_{0.95}} \\ 
\dfrac{\partial \widehat{\beta }}{\partial \widehat{\xi }_{0.05}} &
\dfrac{\partial \widehat{\beta }}{\partial \widehat{\xi }_{0.25}} &
\dfrac{\partial \widehat{\beta }}{\partial \widehat{\xi }_{0.50}} &
\dfrac{\partial \widehat{\beta }}{\partial \widehat{\xi }_{0.75}} &
\dfrac{\partial \widehat{\beta }}{\partial \widehat{\xi }_{0.95}} \\
\dfrac{\partial \widehat{\gamma }}{\partial \widehat{\xi }_{0.05}} &
\dfrac{\partial \widehat{\gamma }}{\partial \widehat{\xi }_{0.25}} &
\dfrac{\partial \widehat{\gamma }}{\partial \widehat{\xi }_{0.50}} &
\dfrac{\partial \widehat{\gamma }}{\partial \widehat{\xi }_{0.75}} &
\dfrac{\partial \widehat{\gamma }}{\partial \widehat{\xi }_{0.95}} \\
\dfrac{\partial \widehat{\delta }}{\partial \widehat{\xi }_{0.05}} &
\dfrac{\partial \widehat{\delta }}{\partial \widehat{\xi }_{0.25}} &
\dfrac{\partial \widehat{\delta }}{\partial \widehat{\xi }_{0.50}} &
\dfrac{\partial \widehat{\delta }}{\partial \widehat{\xi }_{0.75}} &
\dfrac{\partial \widehat{\delta }}{\partial \widehat{\xi }_{0.95}}
\end{array}
\right) ^{\prime },
\label{eq:est:15}
\end{equation}
then following the same approach taken by in \cite{ci:M86} using the Multivariate Delta Theorem (see \cite{ci:S80}), we obtain 
\begin{equation}
\sqrt{n}\left( \left( 
\begin{array}{c}
\widehat{\alpha } \\ 
\widehat{\beta } \\
\widehat{\gamma} \\
\widehat{\delta}
\end{array}
\right) -\left( 
\begin{array}{c}
\alpha \\ 
\beta \\
\gamma \\
\delta
\end{array}
\right) \right) \overset{d}{\longrightarrow }N\left( 0,D^{\prime }\Sigma_M
D\right).
\label{eq:est:16}
\end{equation}

A general analytic formula is not available for the calculation of the
partial derivatives in $\left( \ref{eq:est:15}\right) $. It is suggested
in \cite{ci:M86} that the partial derivatives can be estimated ``by means of
small perturbations of the population quantiles'', but no specific
recommendations regarding the size of these perturbations are given. To
limit the scope of our investigation into this matter, we restrict ourselves
to perturbations given by 
\begin{equation}
\Delta \xi =\dfrac{\widehat{\xi }_{0.75}-\widehat{\xi }_{0.25}}{C}
\label{eq:est:17}
\end{equation}%
for some $C>0$ and assume that the same perturbation is applied to each
quantile estimator. Let $\widehat{\alpha }%
_{p}^{+} $ be the estimate of $\alpha $ derived from the set of quantiles
where $\widehat{\xi }_{p}$ is replaced by $\widehat{\xi }_{p}+\Delta \xi $
and $\widehat{\alpha }_{p}^{-}$ be the estimate of $\alpha $ derived from
the set of quantile estimates where $\widehat{\xi }_{p}$ is replaced by $\widehat{\xi 
}_{p}-\Delta \xi .$ Similarly, we define $\widehat{\beta }_{p}^{+},~\widehat{%
\beta }_{p}^{-},$ etc. Our estimate for $\dfrac{\partial \widehat{\alpha }}{%
\partial \widehat{\xi} _{p}}$ is defined to be%
\begin{equation}
\widehat{\dfrac{\partial \widehat{\alpha }}{\partial \widehat{\xi }_{p}}}=%
\dfrac{\widehat{\alpha }_{p}^{+}-\widehat{\alpha }_{p}^{-}}{2\Delta \xi }
\label{eq:SMA:SPX:23}
\end{equation}%
with similar definitions for $\widehat{\dfrac{\partial \widehat{\beta }}{%
\partial \widehat{\xi} _{p}}},\widehat{\dfrac{\partial \widehat{\gamma }}{\partial \widehat{\xi}
_{p}}}$ and $\widehat{\dfrac{\partial \widehat{\delta }}{\partial \widehat{\xi} _{p}}}%
. $

Estimates for each of the partial derivative estimators were calculated for
various stable distributions. Examples of these calculations are presented
in Figure \ref{fig:est:1} for values of $C$ between 50 and 1000. The
optimal choice for $C$ is not obvious, given we do not have any true values
for the partial derivatives. However, in general the value of the partial
derivative estimates does not change greatly for $C$ between 50 and 1000$.$
A\ slightly lower value of $C\ $and hence slightly larger perturbation can
help to smooth the partial derivatives and avoid occasional numerical
abberations. Throughout this paper we use $C=400$, to calculate the partial
derivative estimates. In Figure \ref{fig:est:1}, the estimates
calculated using $C=400$ are those indicated by fourth $^{\prime }\ast
^{\prime }$ from the left.

\begin{figure}[h] \centering
\begin{tabular}{cc}
\includegraphics[width=6.5cm]{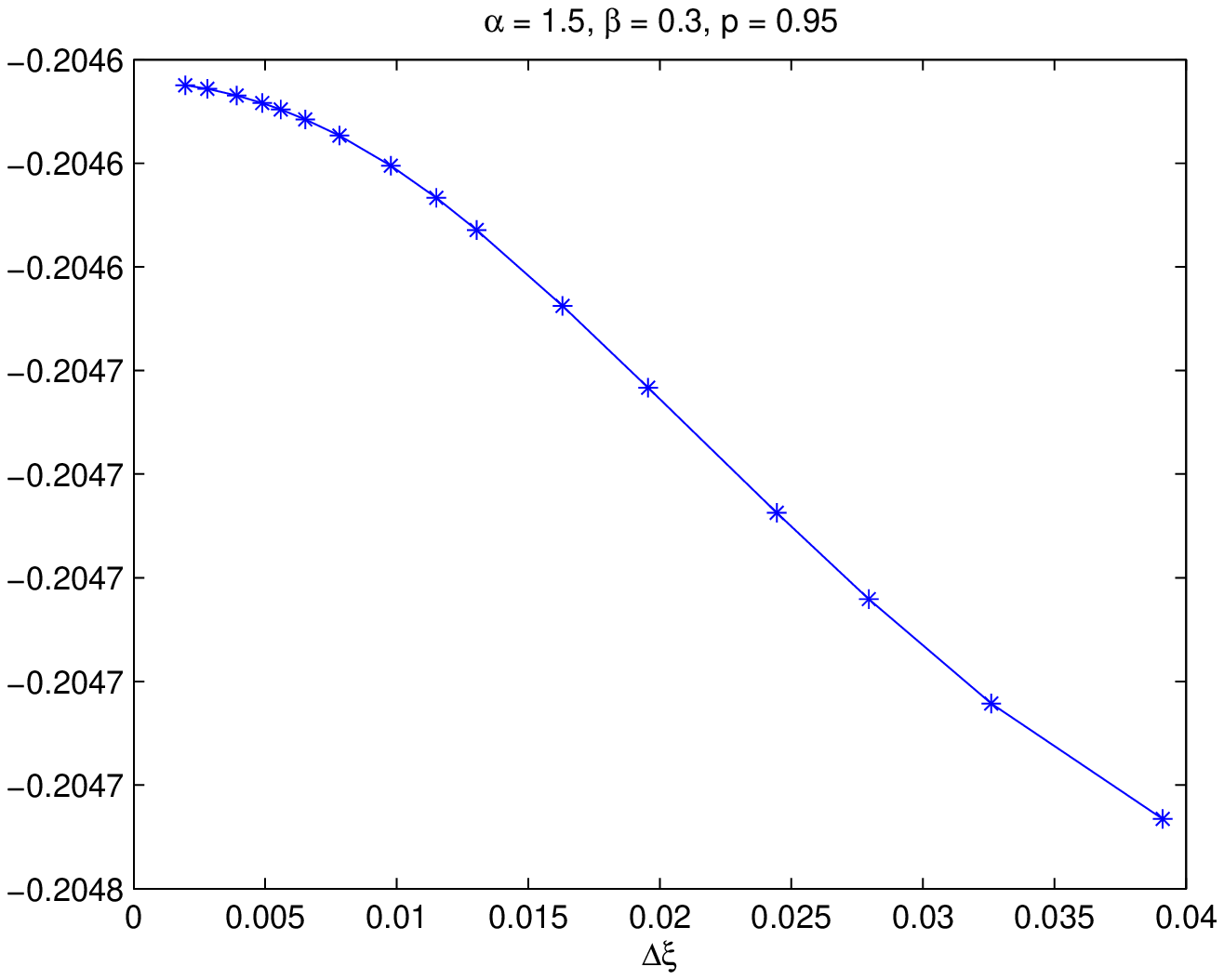} &
\includegraphics[width=6.5cm]{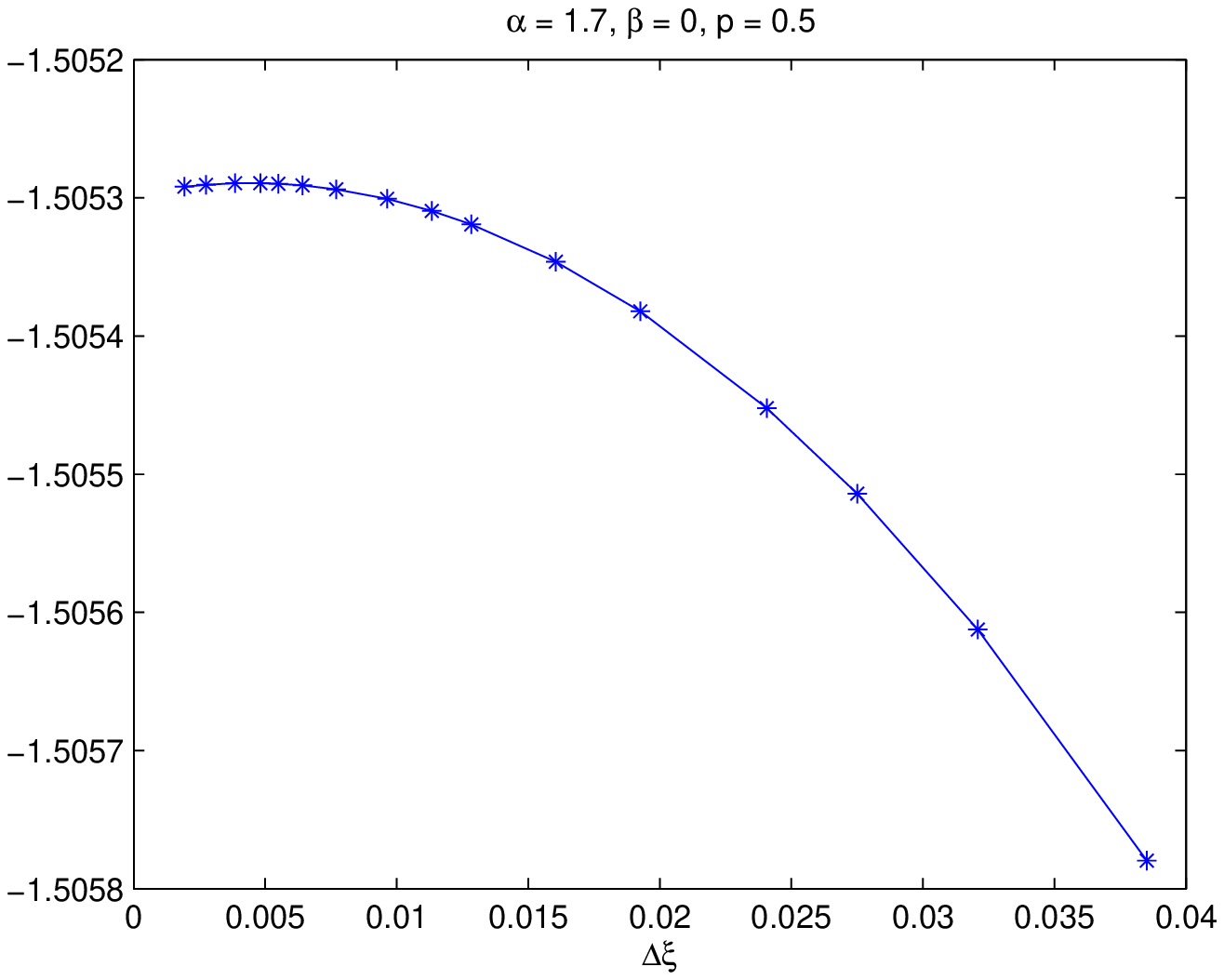} \\
\footnotesize{(a)}&
\footnotesize{(b)}\\
\end{tabular}
\caption{Estimates for (a) 
$\protect\dfrac{\protect\partial \protect\widehat{\protect\alpha }}
{\protect\partial \protect\widehat{\protect\xi }_{0.95}}$ where $\protect\alpha = 1.5$ and $\protect\beta = 0.3$
and (b)
$\protect\dfrac{\protect\partial \protect\widehat{\protect\beta }}
{\protect\partial \protect\widehat{\protect\xi }_{0.50}}$ where $\protect\alpha = 1.7$ and $\protect\beta = 0.0$.}%
\label{fig:est:1}%
\end{figure}%

With some minor modifications by the author, the \textsc{matlab}\ package 
\textsc{stbl}\_\textsc{code}\ was used throughout this paper to generate
sequences of stable random variable, calculate values of the stable density,
distribution and quantile functions. To implement stable distribution
parameter estimation, a lookup table for $\psi _{1}$ and $\psi _{2}$ with $%
184$ values of $v_{\alpha }$ and $86$ values of $v_{\beta }$ was generated.
Interpolation is used to calculate the values of $\psi _{1}$ and $\psi _{2}$
for those values of $v_{\alpha }$ and $v_{\beta }$ which do not exactly
match the lookup table values. Spline interpolation is used in preference to
linear interpolation, except for $\alpha $ close to $2,$ where spline
interpolation occasionally performs poorly. All partial derivatives in
Figure \ref{fig:est:1} were calculated using spline interpolation. If 
linear interpolation was used to calculate the derivatives in Figure \ref%
{fig:est:1}, then the resulting plots would show discontinuities in the
first derivative at points where the values of $v_{\alpha }$ and $v_{\beta }$
move between cells in the lookup tables.

\section{Simulation}
\label{sec:sim}
In this section we present the results of some simulations which demonstrate
the use of the methods described in this paper for the estimation of the
stable distribution parameters of a $\textsc{sma}$(1) process. For selected set of values $%
\alpha $, $\beta $, and $\theta _{1}$ a simulation is run where 2,000
realisations of the process, each of length $n=720$, are generated
. The parameters $\gamma_0 =2$ and $\delta_0 =1$ are fixed for all simulations.
Estimates for the parameters $\alpha $, $\beta$, $\gamma$ and $\delta$ are
calculated for each realisation. The mean and variance of these estimates
across all realisations of a particular simulation are then compared with
the true parameter values and the asymptotic variance of the estimators. The
results for $\alpha $, $\beta$, $\gamma$ and $\delta$ are reported
in Tables \ref{tab:sim:1}, \ref{tab:sim:2}, \ref%
{tab:sim:3} and \ref{tab:sim:4} respectively.

\begin{table}[tbp] \centering%
\begin{tabular}{|c|c|c|c|c|c|c|c|}
\hline\hline
&  & \multicolumn{2}{|c}{$\theta _{1}=0.0$} & \multicolumn{2}{|c|}{$\theta
_{1}=0.2$} & \multicolumn{2}{|c|}{$\theta _{1}=0.4$} \\ \hline
$\alpha $ & $\beta $ & \multicolumn{1}{||c|}{$\text{{\footnotesize (i)}}$} & 
$\text{{\footnotesize (ii)}}$ & $\text{{\footnotesize (i)}}$ & $\text{%
{\footnotesize (ii)}}$ & $\text{{\footnotesize (i)}}$ & $\text{%
{\footnotesize (ii)}}$ \\ \hline\hline
$1.2$ & $0.0$ & \multicolumn{1}{|r|}{$%
\begin{array}{c}
\text{{\small 1.195}} \\ 
\text{{\small (0.061)}}%
\end{array}%
$} & \multicolumn{1}{|r|}{$%
\begin{array}{c}
\text{{\small 2.682}} \\ 
\text{{\small [2.555]}}%
\end{array}%
$} & \multicolumn{1}{|r|}{$%
\begin{array}{c}
\text{{\small 1.195}} \\ 
\text{{\small (0.063)}}%
\end{array}%
$} & \multicolumn{1}{|r|}{$%
\begin{array}{c}
\text{{\small 2.891}} \\ 
\text{{\small [2.740]}}%
\end{array}%
$} & \multicolumn{1}{|r|}{$%
\begin{array}{c}
\text{{\small 1.194}} \\ 
\text{{\small (0.068)}}%
\end{array}%
$} & \multicolumn{1}{|r|}{$%
\begin{array}{c}
\text{{\small 3.322}} \\ 
\text{{\small [3.200]}}%
\end{array}%
$} \\ \hline
$1.2$ & $0.2$ & \multicolumn{1}{|r|}{$%
\begin{array}{c}
\text{{\small 1.196}} \\ 
\text{{\small (0.062)}}%
\end{array}%
$} & \multicolumn{1}{|r|}{$%
\begin{array}{c}
\text{{\small 2.746}} \\ 
\text{{\small [2.833]}}%
\end{array}%
$} & \multicolumn{1}{|r|}{$%
\begin{array}{c}
\text{{\small 1.199}} \\ 
\text{{\small (0.067)}}%
\end{array}%
$} & \multicolumn{1}{|r|}{$%
\begin{array}{c}
\text{{\small 3.224}} \\ 
\text{{\small [3.040]}}%
\end{array}%
$} & \multicolumn{1}{|r|}{$%
\begin{array}{c}
\text{{\small 1.198}} \\ 
\text{{\small (0.072)}}%
\end{array}%
$} & \multicolumn{1}{|r|}{$%
\begin{array}{c}
\text{{\small 3.702}} \\ 
\text{{\small [3.568]}}%
\end{array}%
$} \\ \hline
$1.2$ & $0.5$ & \multicolumn{1}{|r|}{$%
\begin{array}{c}
\text{{\small 1.202}} \\ 
\text{{\small (0.076)}}%
\end{array}%
$} & \multicolumn{1}{|r|}{$%
\begin{array}{c}
\text{{\small 4.134}} \\ 
\text{{\small [3.975]}}%
\end{array}%
$} & \multicolumn{1}{|r|}{$%
\begin{array}{c}
\text{{\small 1.201}} \\ 
\text{{\small (0.076)}}%
\end{array}%
$} & \multicolumn{1}{|r|}{$%
\begin{array}{c}
\text{{\small 4.204}} \\ 
\text{{\small [4.257]}}%
\end{array}%
$} & \multicolumn{1}{|r|}{$%
\begin{array}{c}
\text{{\small 1.200}} \\ 
\text{{\small (0.084)}}%
\end{array}%
$} & \multicolumn{1}{|r|}{$%
\begin{array}{c}
\text{{\small 5.074}} \\ 
\text{{\small [5.058]}}%
\end{array}%
$} \\ \hline
$1.5$ & $0.0$ & \multicolumn{1}{|r|}{$%
\begin{array}{c}
\text{{\small 1.503}} \\ 
\text{{\small (0.076)}}%
\end{array}%
$} & \multicolumn{1}{|r|}{$%
\begin{array}{c}
\text{{\small 4.105}} \\ 
\text{{\small [3.852]}}%
\end{array}%
$} & \multicolumn{1}{|r|}{$%
\begin{array}{c}
\text{{\small 1.504}} \\ 
\text{{\small (0.079)}}%
\end{array}%
$} & \multicolumn{1}{|r|}{$%
\begin{array}{c}
\text{{\small 4.472}} \\ 
\text{{\small [3.984]}}%
\end{array}%
$} & \multicolumn{1}{|r|}{$%
\begin{array}{c}
\text{{\small 1.502}} \\ 
\text{{\small (0.080)}}%
\end{array}%
$} & \multicolumn{1}{|r|}{$%
\begin{array}{c}
\text{{\small 4.604}} \\ 
\text{{\small [4.348]}}%
\end{array}%
$} \\ \hline
$1.5$ & $0.2$ & \multicolumn{1}{|r|}{$%
\begin{array}{c}
\text{{\small 1.504}} \\ 
\text{{\small (0.078)}}%
\end{array}%
$} & \multicolumn{1}{|r|}{$%
\begin{array}{c}
\text{{\small 4.346}} \\ 
\text{{\small [4.076]}}%
\end{array}%
$} & \multicolumn{1}{|r|}{$%
\begin{array}{c}
\text{{\small 1.506}} \\ 
\text{{\small (0.081)}}%
\end{array}%
$} & \multicolumn{1}{|r|}{$%
\begin{array}{c}
\text{{\small 4.693}} \\ 
\text{{\small [4.217]}}%
\end{array}%
$} & \multicolumn{1}{|r|}{$%
\begin{array}{c}
\text{{\small 1.505}} \\ 
\text{{\small (0.087)}}%
\end{array}%
$} & \multicolumn{1}{|r|}{$%
\begin{array}{c}
\text{{\small 5.466}} \\ 
\text{{\small [4.611]}}%
\end{array}%
$} \\ \hline
$1.5$ & $0.5$ & \multicolumn{1}{|r|}{$%
\begin{array}{c}
\text{{\small 1.505}} \\ 
\text{{\small (0.087)}}%
\end{array}%
$} & \multicolumn{1}{|r|}{$%
\begin{array}{c}
\text{{\small 5.384}} \\ 
\text{{\small [5.207]}}%
\end{array}%
$} & \multicolumn{1}{|r|}{$%
\begin{array}{c}
\text{{\small 1.506}} \\ 
\text{{\small (0.091)}}%
\end{array}%
$} & \multicolumn{1}{|r|}{$%
\begin{array}{c}
\text{{\small 5.907}} \\ 
\text{{\small [5.384]}}%
\end{array}%
$} & \multicolumn{1}{|r|}{$%
\begin{array}{c}
\text{{\small 1.506}} \\ 
\text{{\small (0.094)}}%
\end{array}%
$} & \multicolumn{1}{|r|}{$%
\begin{array}{c}
\text{{\small 6.375}} \\ 
\text{{\small [5.919]}}%
\end{array}%
$} \\ \hline
$1.8$ & $0.0$ & \multicolumn{1}{|r|}{$%
\begin{array}{c}
\text{{\small 1.808}} \\ 
\text{{\small (0.108)}}%
\end{array}%
$} & \multicolumn{1}{|r|}{$%
\begin{array}{c}
\text{{\small 8.389}} \\ 
\text{{\small [9.471]}}%
\end{array}%
$} & \multicolumn{1}{|r|}{$%
\begin{array}{c}
\text{{\small 1.809}} \\ 
\text{{\small (0.107)}}%
\end{array}%
$} & \multicolumn{1}{|r|}{$%
\begin{array}{c}
\text{{\small 8.160}} \\ 
\text{{\small [9.544]}}%
\end{array}%
$} & \multicolumn{1}{|r|}{$%
\begin{array}{c}
\text{{\small 1.809}} \\ 
\text{{\small (0.109)}}%
\end{array}%
$} & \multicolumn{1}{|r|}{$%
\begin{array}{c}
\text{{\small 8.495}} \\ 
\text{{\small [9.783]}}%
\end{array}%
$} \\ \hline
$1.8$ & $0.2$ & \multicolumn{1}{|r|}{$%
\begin{array}{c}
\text{{\small 1.810}} \\ 
\text{{\small (0.109)}}%
\end{array}%
$} & \multicolumn{1}{|r|}{$%
\begin{array}{c}
\text{{\small 8.515}} \\ 
\text{{\small [9.536]}}%
\end{array}%
$} & \multicolumn{1}{|r|}{$%
\begin{array}{c}
\text{{\small 1.808}} \\ 
\text{{\small (0.110)}}%
\end{array}%
$} & \multicolumn{1}{|r|}{$%
\begin{array}{c}
\text{{\small 8.678}} \\ 
\text{{\small [9.611]}}%
\end{array}%
$} & \multicolumn{1}{|r|}{$%
\begin{array}{c}
\text{{\small 1.806}} \\ 
\text{{\small (0.109)}}%
\end{array}%
$} & \multicolumn{1}{|r|}{$%
\begin{array}{c}
\text{{\small 8.614}} \\ 
\text{{\small [9.853]}}%
\end{array}%
$} \\ \hline
$1.8$ & $0.5$ & \multicolumn{1}{|r|}{$%
\begin{array}{c}
\text{{\small 1.809}} \\ 
\text{{\small (0.106)}}%
\end{array}%
$} & \multicolumn{1}{|r|}{$%
\begin{array}{c}
\text{{\small 8.083}} \\ 
\text{{\small [9.902]}}%
\end{array}%
$} & \multicolumn{1}{|r|}{$%
\begin{array}{c}
\text{{\small 1.808}} \\ 
\text{{\small (0.108)}}%
\end{array}%
$} & \multicolumn{1}{|r|}{$%
\begin{array}{c}
\text{{\small 8.353}} \\ 
\text{{\small [9.981]}}%
\end{array}%
$} & \multicolumn{1}{|r|}{$%
\begin{array}{c}
\text{{\small 1.809}} \\ 
\text{{\small (0.110)}}%
\end{array}%
$} & \multicolumn{1}{|r|}{$%
\begin{array}{c}
\text{{\small 8.669}} \\ 
\text{{\small [10.24]}}%
\end{array}%
$} \\ \hline
\end{tabular}%
\caption{Simulation results of the estimation of $\protect\alpha$ from selected 
$\protect\textsc{sma}$(1) processes. Reported in
this table for each process are 
(i) the mean and standard deviation, in ( ), of $\protect\widehat{\protect\alpha}$ across all realisations and 
(ii) the variance of $\protect\widehat{\protect\alpha}$ across all realisations multiplied by the sample size $n$ 
and the true asymptotic variance, in [ ].}
\label{tab:sim:1}%
\end{table}%

In each case the mean value of the estimator across all realisations is
within one standard deviation of the true parameter value and is generally
much closer than that. The normalised variance (i.e. the variance multiplied
by the sample size) across all realisations is reasonably close to the
asymptotic variance.

The normalised variance of $\widehat{\alpha }$ where $\alpha =1.8$ appears
to be slightly less than the asymptotic variance. This is due to the
truncation of all $\widehat{\alpha }$ estimates into the range $\left( 0,2%
\right] .$ A similar effect is seen with $\widehat{\beta }$ estimates where $\alpha =1.8$ and $\beta_0 =0.5$. Estimates of $\widehat{%
\beta }$ where $\alpha =1.8$ are the least precise. This
is to be expected as the asymptotic variance of $\widehat{\beta }$ increases to $\infty $ as $\alpha $ increases to 2.

For each of the selected \textsc{sma}(1) processes and for each of the
estimators $\widehat{\alpha }$, $\widehat{\beta }$, $%
\widehat{\gamma }$ and $\widehat{\delta }$ the asymptotic variance of the estimator is higher for $\theta
_{1}=0.2$ than for $\theta _{1}=0.0$ and higher still for $\theta _{1}=0.4.$
The effect of increases in $\theta _{1}$ on the asymptotic variance of the
estimators $\widehat{\alpha }$, $\widehat{\beta }$, $%
\widehat{\gamma }$ and $\widehat{\delta }$ appears to decrease as $\alpha $ increases and is more
significant for $\widehat{\gamma }$ and $\widehat{\delta }%
$ than for $\widehat{\alpha }$ and $\widehat{\beta }$. From additional simulation results not included in
this paper, we observe that the asymptotic variance of $\widehat{\gamma }$ appears symmetric in $\theta _{1}$ about zero, however
that does not appear to be the case for $\widehat{\alpha }$, $\widehat{\beta 
}$ and $\widehat{\delta }$ where more
complicated relationships exist between the asymptotic variances and the
parameter values.

\begin{table}[tbp] \centering%
\begin{tabular}{|c|c|c|c|c|c|c|c|}
\hline\hline
&  & \multicolumn{2}{|c}{$\theta _{1}=0.0$} & \multicolumn{2}{|c|}{$\theta
_{1}=0.2$} & \multicolumn{2}{|c|}{$\theta _{1}=0.4$} \\ \hline
$\alpha $ & $\beta $ & \multicolumn{1}{||c|}{$\text{{\footnotesize (i)}}$} & 
$\text{{\footnotesize (ii)}}$ & $\text{{\footnotesize (i)}}$ & $\text{%
{\footnotesize (ii)}}$ & $\text{{\footnotesize (i)}}$ & $\text{%
{\footnotesize (ii)}}$ \\ \hline\hline
$1.2$ & $0.0$ & \multicolumn{1}{|r|}{$%
\begin{array}{c}
\text{{\small 0.000}} \\ 
\text{{\small (0.105)}}%
\end{array}%
$} & \multicolumn{1}{|r|}{$%
\begin{array}{c}
\text{{\small 7.961}} \\ 
\text{{\small [8.684]}}%
\end{array}%
$} & \multicolumn{1}{|r|}{$%
\begin{array}{c}
\text{{\small 0.002}} \\ 
\text{{\small (0.121)}}%
\end{array}%
$} & \multicolumn{1}{|r|}{$%
\begin{array}{c}
\text{{\small 10.46}} \\ 
\text{{\small [10.91]}}%
\end{array}%
$} & \multicolumn{1}{|r|}{$%
\begin{array}{c}
\text{{\small 0.001}} \\ 
\text{{\small (0.131)}}%
\end{array}%
$} & \multicolumn{1}{|r|}{$%
\begin{array}{c}
\text{{\small 12.30}} \\ 
\text{{\small [12.94]}}%
\end{array}%
$} \\ \hline
$1.2$ & $0.2$ & \multicolumn{1}{|r|}{$%
\begin{array}{c}
\text{{\small 0.193}} \\ 
\text{{\small (0.101)}}%
\end{array}%
$} & \multicolumn{1}{|r|}{$%
\begin{array}{c}
\text{{\small 7.408}} \\ 
\text{{\small [7.677]}}%
\end{array}%
$} & \multicolumn{1}{|r|}{$%
\begin{array}{c}
\text{{\small 0.196}} \\ 
\text{{\small (0.113)}}%
\end{array}%
$} & \multicolumn{1}{|r|}{$%
\begin{array}{c}
\text{{\small 9.230}} \\ 
\text{{\small [9.657]}}%
\end{array}%
$} & \multicolumn{1}{|r|}{$%
\begin{array}{c}
\text{{\small 0.194}} \\ 
\text{{\small (0.124)}}%
\end{array}%
$} & \multicolumn{1}{|r|}{$%
\begin{array}{c}
\text{{\small 11.03}} \\ 
\text{{\small [11.44]}}%
\end{array}%
$} \\ \hline
$1.2$ & $0.5$ & \multicolumn{1}{|r|}{$%
\begin{array}{c}
\text{{\small 0.498}} \\ 
\text{{\small (0.082)}}%
\end{array}%
$} & \multicolumn{1}{|r|}{$%
\begin{array}{c}
\text{{\small 4.866}} \\ 
\text{{\small [4.223]}}%
\end{array}%
$} & \multicolumn{1}{|r|}{$%
\begin{array}{c}
\text{{\small 0.495}} \\ 
\text{{\small (0.089)}}%
\end{array}%
$} & \multicolumn{1}{|r|}{$%
\begin{array}{c}
\text{{\small 5.751}} \\ 
\text{{\small [5.199]}}%
\end{array}%
$} & \multicolumn{1}{|r|}{$%
\begin{array}{c}
\text{{\small 0.494}} \\ 
\text{{\small (0.098)}}%
\end{array}%
$} & \multicolumn{1}{|r|}{$%
\begin{array}{c}
\text{{\small 6.852}} \\ 
\text{{\small [6.046]}}%
\end{array}%
$} \\ \hline
$1.5$ & $0.0$ & \multicolumn{1}{|r|}{$%
\begin{array}{c}
\text{{\small -0.001}} \\ 
\text{{\small (0.127)}}%
\end{array}%
$} & \multicolumn{1}{|r|}{$%
\begin{array}{c}
\text{{\small 11.60}} \\ 
\text{{\small [11.67]}}%
\end{array}%
$} & \multicolumn{1}{|r|}{$%
\begin{array}{c}
\text{{\small 0.001}} \\ 
\text{{\small (0.140)}}%
\end{array}%
$} & \multicolumn{1}{|r|}{$%
\begin{array}{c}
\text{{\small 14.10}} \\ 
\text{{\small [13.27]}}%
\end{array}%
$} & \multicolumn{1}{|r|}{$%
\begin{array}{c}
\text{{\small 0.000}} \\ 
\text{{\small (0.151)}}%
\end{array}%
$} & \multicolumn{1}{|r|}{$%
\begin{array}{c}
\text{{\small 16.44}} \\ 
\text{{\small [15.52]}}%
\end{array}%
$} \\ \hline
$1.5$ & $0.2$ & \multicolumn{1}{|r|}{$%
\begin{array}{c}
\text{{\small 0.202}} \\ 
\text{{\small (0.140)}}%
\end{array}%
$} & \multicolumn{1}{|r|}{$%
\begin{array}{c}
\text{{\small 14.06}} \\ 
\text{{\small [11.16]}}%
\end{array}%
$} & \multicolumn{1}{|r|}{$%
\begin{array}{c}
\text{{\small 0.204}} \\ 
\text{{\small (0.139)}}%
\end{array}%
$} & \multicolumn{1}{|r|}{$%
\begin{array}{c}
\text{{\small 13.88}} \\ 
\text{{\small [12.60]}}%
\end{array}%
$} & \multicolumn{1}{|r|}{$%
\begin{array}{c}
\text{{\small 0.203}} \\ 
\text{{\small (0.151)}}%
\end{array}%
$} & \multicolumn{1}{|r|}{$%
\begin{array}{c}
\text{{\small 16.33}} \\ 
\text{{\small [14.63]}}%
\end{array}%
$} \\ \hline
$1.5$ & $0.5$ & \multicolumn{1}{|r|}{$%
\begin{array}{c}
\text{{\small 0.525}} \\ 
\text{{\small (0.147)}}%
\end{array}%
$} & \multicolumn{1}{|r|}{$%
\begin{array}{c}
\text{{\small 15.44}} \\ 
\text{{\small [11.44]}}%
\end{array}%
$} & \multicolumn{1}{|r|}{$%
\begin{array}{c}
\text{{\small 0.522}} \\ 
\text{{\small (0.152)}}%
\end{array}%
$} & \multicolumn{1}{|r|}{$%
\begin{array}{c}
\text{{\small 16.70}} \\ 
\text{{\small [12.24]}}%
\end{array}%
$} & \multicolumn{1}{|r|}{$%
\begin{array}{c}
\text{{\small 0.520}} \\ 
\text{{\small (0.156)}}%
\end{array}%
$} & \multicolumn{1}{|r|}{$%
\begin{array}{c}
\text{{\small 17.57}} \\ 
\text{{\small [13.45]}}%
\end{array}%
$} \\ \hline
$1.8$ & $0.0$ & \multicolumn{1}{|r|}{$%
\begin{array}{c}
-\text{{\small 0.008}} \\ 
\text{{\small (0.367)}}%
\end{array}%
$} & \multicolumn{1}{|r|}{$%
\begin{array}{c}
\text{{\small 97.20}} \\ 
\text{{\small [53.62]}}%
\end{array}%
$} & \multicolumn{1}{|r|}{$%
\begin{array}{c}
\text{{\small 0.008}} \\ 
\text{{\small (0.369)}}%
\end{array}%
$} & \multicolumn{1}{|r|}{$%
\begin{array}{c}
\text{{\small 97.76}} \\ 
\text{{\small [55.74]}}%
\end{array}%
$} & \multicolumn{1}{|r|}{$%
\begin{array}{c}
\text{{\small 0.000}} \\ 
\text{{\small (0.375)}}%
\end{array}%
$} & \multicolumn{1}{|r|}{$%
\begin{array}{c}
\text{{\small 101.5}} \\ 
\text{{\small [60.59]}}%
\end{array}%
$} \\ \hline
$1.8$ & $0.2$ & \multicolumn{1}{|r|}{$%
\begin{array}{c}
\text{{\small 0.223}} \\ 
\text{{\small (0.367)}}%
\end{array}%
$} & \multicolumn{1}{|r|}{$%
\begin{array}{c}
\text{{\small 97.40}} \\ 
\text{{\small [63.83]}}%
\end{array}%
$} & \multicolumn{1}{|r|}{$%
\begin{array}{c}
\text{{\small 0.202}} \\ 
\text{{\small (0.381)}}%
\end{array}%
$} & \multicolumn{1}{|r|}{$%
\begin{array}{c}
\text{{\small 104.4}} \\ 
\text{{\small [65.85]}}%
\end{array}%
$} & \multicolumn{1}{|r|}{$%
\begin{array}{c}
\text{{\small 0.186}} \\ 
\text{{\small (0.380)}}%
\end{array}%
$} & \multicolumn{1}{|r|}{$%
\begin{array}{c}
\text{{\small 104.1}} \\ 
\text{{\small [70.58]}}%
\end{array}%
$} \\ \hline
$1.8$ & $0.5$ & \multicolumn{1}{|r|}{$%
\begin{array}{c}
\text{{\small 0.488}} \\ 
\text{{\small (0.334)}}%
\end{array}%
$} & \multicolumn{1}{|r|}{$%
\begin{array}{c}
\text{{\small 80.51}} \\ 
\text{{\small [118.5]}}%
\end{array}%
$} & \multicolumn{1}{|r|}{$%
\begin{array}{c}
\text{{\small 0.488}} \\ 
\text{{\small (0.347)}}%
\end{array}%
$} & \multicolumn{1}{|r|}{$%
\begin{array}{c}
\text{{\small 86.49}} \\ 
\text{{\small [120.0]}}%
\end{array}%
$} & \multicolumn{1}{|r|}{$%
\begin{array}{c}
\text{{\small 0.484}} \\ 
\text{{\small (0.353)}}%
\end{array}%
$} & \multicolumn{1}{|r|}{$%
\begin{array}{c}
\text{{\small 89.85}} \\ 
\text{{\small [124.1]}}%
\end{array}%
$} \\ \hline
\end{tabular}%
\caption{Simulation results of the estimation of $\protect\beta$ 
from selected $\protect\textsc{sma}$(1) processes.
Reported in this table for each process are 
(i) the mean and standard deviation, in ( ), of $\protect\widehat{\protect\beta}$ 
across all realisations and 
(ii) the variance of $\protect\widehat{\protect\beta}$ across all realisations 
multiplied by the sample size $n$ and the true asymptotic variance, in [ ].}%
\label{tab:sim:2}%
\end{table}%

These simulations provide some confidence that the estimators discussed in
this paper, are an unbiased method for the estimation of stable
distribution parameters from a \textsc{sma}(1) process and that the
asymptotic variance provides a good approximation for estimator
variance at sample sizes equal to 720.

\begin{table}[tbp] \centering%
\begin{tabular}{|c|c||c|c|c|c|c|c|}
\hline\hline
&  & \multicolumn{3}{|c}{$\theta _{1}=0.0$} & \multicolumn{3}{|c|}{$\theta
_{1}=0.4$} \\ \hline
$\alpha $ & $\beta $ & $\gamma$ & $\text{{\footnotesize %
(i)}}$ & $\text{{\footnotesize (ii)}}$ & $\gamma$ & $%
\text{{\footnotesize (i)}}$ & $\text{{\footnotesize (ii)}}$ \\ \hline\hline
$1.2$ & $0.0$ & $\text{{\small 2.000}}$ & \multicolumn{1}{|r|}{$%
\begin{array}{c}
\text{{\small 1.991}} \\ 
\text{{\small (0.104)}}%
\end{array}%
$} & \multicolumn{1}{|r|}{$%
\begin{array}{c}
\text{{\small 7.751}} \\ 
\text{{\small [7.983]}}%
\end{array}%
$} & $\text{{\small 2.}}${\small 541} & \multicolumn{1}{|r|}{$%
\begin{array}{c}
\text{{\small 2.523}} \\ 
\text{{\small (0.153)}}%
\end{array}%
$} & \multicolumn{1}{|r|}{$%
\begin{array}{c}
\text{{\small 16.82}} \\ 
\text{{\small [16.32]}}%
\end{array}%
$} \\ \hline
$1.2$ & $0.2$ & $\text{{\small 2.000}}$ & \multicolumn{1}{|r|}{$%
\begin{array}{c}
\text{{\small 1.992}} \\ 
\text{{\small (0.112)}}%
\end{array}%
$} & \multicolumn{1}{|r|}{$%
\begin{array}{c}
\text{{\small 9.000}} \\ 
\text{{\small [8.648]}}%
\end{array}%
$} & $\text{{\small 2.541}}$ & \multicolumn{1}{|r|}{$%
\begin{array}{c}
\text{{\small 2.535}} \\ 
\text{{\small (0.156)}}%
\end{array}%
$} & \multicolumn{1}{|r|}{$%
\begin{array}{c}
\text{{\small 17.58}} \\ 
\text{{\small [17.52]}}%
\end{array}%
$} \\ \hline
$1.2$ & $0.5$ & $\text{{\small 2.000}}$ & \multicolumn{1}{|r|}{$%
\begin{array}{c}
\text{{\small 2.001}} \\ 
\text{{\small (0.123)}}%
\end{array}%
$} & \multicolumn{1}{|r|}{$%
\begin{array}{c}
\text{{\small 10.83}} \\ 
\text{{\small [10.80]}}%
\end{array}%
$} & $\text{{\small 2.541}}$ & \multicolumn{1}{|r|}{$%
\begin{array}{c}
\text{{\small 2.543}} \\ 
\text{{\small (0.171)}}%
\end{array}%
$} & \multicolumn{1}{|r|}{$%
\begin{array}{c}
\text{{\small 21.14}} \\ 
\text{{\small [21.15]}}%
\end{array}%
$} \\ \hline
$1.5$ & $0.0$ & $\text{{\small 2.000}}$ & \multicolumn{1}{|r|}{$%
\begin{array}{c}
\text{{\small 2.000}} \\ 
\text{{\small (0.097)}}%
\end{array}%
$} & \multicolumn{1}{|r|}{$%
\begin{array}{c}
\text{{\small 6.822}} \\ 
\text{{\small [6.553]}}%
\end{array}%
$} & $\text{{\small 2.325}}$ & \multicolumn{1}{|r|}{$%
\begin{array}{c}
\text{{\small 2.322}} \\ 
\text{{\small (0.120)}}%
\end{array}%
$} & \multicolumn{1}{|r|}{$%
\begin{array}{c}
\text{{\small 10.32}} \\ 
\text{{\small [10.45]}}%
\end{array}%
$} \\ \hline
$1.5$ & $0.2$ & $\text{{\small 2.000}}$ & \multicolumn{1}{|r|}{$%
\begin{array}{c}
\text{{\small 1.997}} \\ 
\text{{\small (0.094)}}%
\end{array}%
$} & \multicolumn{1}{|r|}{$%
\begin{array}{c}
\text{{\small 6.290}} \\ 
\text{{\small [6.633]}}%
\end{array}%
$} & $\text{{\small 2.325}}$ & \multicolumn{1}{|r|}{$%
\begin{array}{c}
\text{{\small 2.318}} \\ 
\text{{\small (0.120)}}%
\end{array}%
$} & \multicolumn{1}{|r|}{$%
\begin{array}{c}
\text{{\small 10.37}} \\ 
\text{{\small [10.54]}}%
\end{array}%
$} \\ \hline
$1.5$ & $0.5$ & $\text{{\small 2.000}}$ & \multicolumn{1}{|r|}{$%
\begin{array}{c}
\text{{\small 2.000}} \\ 
\text{{\small (0.098)}}%
\end{array}%
$} & \multicolumn{1}{|r|}{$%
\begin{array}{c}
\text{{\small 6.838}} \\ 
\text{{\small [6.877]}}%
\end{array}%
$} & $\text{{\small 2.325}}$ & \multicolumn{1}{|r|}{$%
\begin{array}{c}
\text{{\small 2.321}} \\ 
\text{{\small (0.125)}}%
\end{array}%
$} & \multicolumn{1}{|r|}{$%
\begin{array}{c}
\text{{\small 11.24}} \\ 
\text{{\small [10.77]}}%
\end{array}%
$} \\ \hline
$1.8$ & $0.0$ & $\text{{\small 2.000}}$ & \multicolumn{1}{|r|}{$%
\begin{array}{c}
\text{{\small 2.000}} \\ 
\text{{\small (0.093)}}%
\end{array}%
$} & \multicolumn{1}{|r|}{$%
\begin{array}{c}
\text{{\small 6.162}} \\ 
\text{{\small [6.272]}}%
\end{array}%
$} & $\text{{\small 2.205}}$ & \multicolumn{1}{|r|}{$%
\begin{array}{c}
\text{{\small 2.204}} \\ 
\text{{\small (0.108)}}%
\end{array}%
$} & \multicolumn{1}{|r|}{$%
\begin{array}{c}
\text{{\small 8.407}} \\ 
\text{{\small [8.482]}}%
\end{array}%
$} \\ \hline
$1.8$ & $0.2$ & $\text{{\small 2.000}}$ & \multicolumn{1}{|r|}{$%
\begin{array}{c}
\text{{\small 2.000}} \\ 
\text{{\small (0.094)}}%
\end{array}%
$} & \multicolumn{1}{|r|}{$%
\begin{array}{c}
\text{{\small 6.417}} \\ 
\text{{\small [6.267]}}%
\end{array}%
$} & $\text{{\small 2.205}}$ & \multicolumn{1}{|r|}{$%
\begin{array}{c}
\text{{\small 2.200}} \\ 
\text{{\small (0.108)}}%
\end{array}%
$} & \multicolumn{1}{|r|}{$%
\begin{array}{c}
\text{{\small 8.378}} \\ 
\text{{\small [8.473]}}%
\end{array}%
$} \\ \hline
$1.8$ & $0.5$ & $\text{{\small 2.000}}$ & \multicolumn{1}{|r|}{$%
\begin{array}{c}
\text{{\small 2.003}} \\ 
\text{{\small (0.093)}}%
\end{array}%
$} & \multicolumn{1}{|r|}{$%
\begin{array}{c}
\text{{\small 6.146}} \\ 
\text{{\small [6.227]}}%
\end{array}%
$} & $\text{{\small 2.205}}$ & \multicolumn{1}{|r|}{$%
\begin{array}{c}
\text{{\small 2.210}} \\ 
\text{{\small (0.111)}}%
\end{array}%
$} & \multicolumn{1}{|r|}{$%
\begin{array}{c}
\text{{\small 8.799}} \\ 
\text{{\small [8.408]}}%
\end{array}%
$} \\ \hline
\end{tabular}%
\caption{Simulation results of the estimation of $\protect\gamma$ 
from selected $\protect\textsc{sma}$(1) processes.
Reported in this table for each process are 
(i) the mean and standard deviation, in ( ), of $\protect\widehat{\protect\gamma}$ 
across all realisations and 
(ii) the variance of $\protect\widehat{\protect\gamma}$ across all realisations 
multiplied by the sample size $n$ and the true asymptotic variance, in [ ].}%
\label{tab:sim:3}%
\end{table}%

\begin{table}[tbp] \centering%
\begin{tabular}{|c|c||c|c|c|c|c|c|}
\hline\hline
&  & \multicolumn{3}{|c}{$\theta _{1}=0.0$} & \multicolumn{3}{|c|}{$\theta
_{1}=0.4$} \\ \hline
$\alpha $ & $\beta $ & $\delta$ & $\text{{\footnotesize %
(i)}}$ & $\text{{\footnotesize (ii)}}$ & $\delta$ & $%
\text{{\footnotesize (i)}}$ & $\text{{\footnotesize (ii)}}$ \\ \hline\hline
$1.2$ & $0.0$ & $\text{{\small 1.000}}$ & \multicolumn{1}{|r|}{$%
\begin{array}{c}
\text{{\small 0.998}} \\ 
\text{{\small (0.130)}}%
\end{array}%
$} & \multicolumn{1}{|r|}{$%
\begin{array}{c}
\text{{\small 12.15}} \\ 
\text{{\small [13.05]}}%
\end{array}%
$} & $\text{{\small 1.}}${\small 400} & \multicolumn{1}{|r|}{$%
\begin{array}{c}
\text{{\small 1.398}} \\ 
\text{{\small (0.194)}}%
\end{array}%
$} & \multicolumn{1}{|r|}{$%
\begin{array}{c}
\text{{\small 27.04}} \\ 
\text{{\small [28.32]}}%
\end{array}%
$} \\ \hline
$1.2$ & $0.2$ & $\text{{\small 1.000}}$ & \multicolumn{1}{|r|}{$%
\begin{array}{c}
\text{{\small 1.003}} \\ 
\text{{\small (0.134)}}%
\end{array}%
$} & \multicolumn{1}{|r|}{$%
\begin{array}{c}
\text{{\small 12.95}} \\ 
\text{{\small [13.19]}}%
\end{array}%
$} & $\text{{\small 1.559}}$ & \multicolumn{1}{|r|}{$%
\begin{array}{c}
\text{{\small 1.558}} \\ 
\text{{\small (0.197)}}%
\end{array}%
$} & \multicolumn{1}{|r|}{$%
\begin{array}{c}
\text{{\small 27.97}} \\ 
\text{{\small [28.62]}}%
\end{array}%
$} \\ \hline
$1.2$ & $0.5$ & $\text{{\small 1.000}}$ & \multicolumn{1}{|r|}{$%
\begin{array}{c}
\text{{\small 1.001}} \\ 
\text{{\small (0.139)}}%
\end{array}%
$} & \multicolumn{1}{|r|}{$%
\begin{array}{c}
\text{{\small 14.00}} \\ 
\text{{\small [14.11]}}%
\end{array}%
$} & $\text{{\small 1.798}}$ & \multicolumn{1}{|r|}{$%
\begin{array}{c}
\text{{\small 1.812}} \\ 
\text{{\small (0.208)}}%
\end{array}%
$} & \multicolumn{1}{|r|}{$%
\begin{array}{c}
\text{{\small 31.15}} \\ 
\text{{\small [30.41]}}%
\end{array}%
$} \\ \hline
$1.5$ & $0.0$ & $\text{{\small 1.000}}$ & \multicolumn{1}{|r|}{$%
\begin{array}{c}
\text{{\small 1.001}} \\ 
\text{{\small (0.143)}}%
\end{array}%
$} & \multicolumn{1}{|r|}{$%
\begin{array}{c}
\text{{\small 14.73}} \\ 
\text{{\small [15.49]}}%
\end{array}%
$} & $\text{{\small 1.400}}$ & \multicolumn{1}{|r|}{$%
\begin{array}{c}
\text{{\small 1.397}} \\ 
\text{{\small (0.200)}}%
\end{array}%
$} & \multicolumn{1}{|r|}{$%
\begin{array}{c}
\text{{\small 28.71}} \\ 
\text{{\small [27.46]}}%
\end{array}%
$} \\ \hline
$1.5$ & $0.2$ & $\text{{\small 1.000}}$ & \multicolumn{1}{|r|}{$%
\begin{array}{c}
\text{{\small 1.005}} \\ 
\text{{\small (0.148)}}%
\end{array}%
$} & \multicolumn{1}{|r|}{$%
\begin{array}{c}
\text{{\small 15.73}} \\ 
\text{{\small [15.61]}}%
\end{array}%
$} & $\text{{\small 1.495}}$ & \multicolumn{1}{|r|}{$%
\begin{array}{c}
\text{{\small 1.500}} \\ 
\text{{\small (0.198)}}%
\end{array}%
$} & \multicolumn{1}{|r|}{$%
\begin{array}{c}
\text{{\small 28.22}} \\ 
\text{{\small [27.68]}}%
\end{array}%
$} \\ \hline
$1.5$ & $0.5$ & $\text{{\small 1.000}}$ & \multicolumn{1}{|r|}{$%
\begin{array}{c}
\text{{\small 0.993}} \\ 
\text{{\small (0.150)}}%
\end{array}%
$} & \multicolumn{1}{|r|}{$%
\begin{array}{c}
\text{{\small 16.27}} \\ 
\text{{\small [16.58]}}%
\end{array}%
$} & $\text{{\small 1.638}}$ & \multicolumn{1}{|r|}{$%
\begin{array}{c}
\text{{\small 1.633}} \\ 
\text{{\small (0.201)}}%
\end{array}%
$} & \multicolumn{1}{|r|}{$%
\begin{array}{c}
\text{{\small 29.09}} \\ 
\text{{\small [29.24]}}%
\end{array}%
$} \\ \hline
$1.8$ & $0.0$ & $\text{{\small 1.000}}$ & \multicolumn{1}{|r|}{$%
\begin{array}{c}
\text{{\small 0.997}} \\ 
\text{{\small (0.160)}}%
\end{array}%
$} & \multicolumn{1}{|r|}{$%
\begin{array}{c}
\text{{\small 18.49}} \\ 
\text{{\small [19.50]}}%
\end{array}%
$} & $\text{{\small 1.400}}$ & \multicolumn{1}{|r|}{$%
\begin{array}{c}
\text{{\small 1.395}} \\ 
\text{{\small (0.202)}}%
\end{array}%
$} & \multicolumn{1}{|r|}{$%
\begin{array}{c}
\text{{\small 29.35}} \\ 
\text{{\small [30.46]}}%
\end{array}%
$} \\ \hline
$1.8$ & $0.2$ & $\text{{\small 1.000}}$ & \multicolumn{1}{|r|}{$%
\begin{array}{c}
\text{{\small 1.010}} \\ 
\text{{\small (0.159)}}%
\end{array}%
$} & \multicolumn{1}{|r|}{$%
\begin{array}{c}
\text{{\small 18.15}} \\ 
\text{{\small [19.61]}}%
\end{array}%
$} & $\text{{\small 1.439}}$ & \multicolumn{1}{|r|}{$%
\begin{array}{c}
\text{{\small 1.454}} \\ 
\text{{\small (0.199)}}%
\end{array}%
$} & \multicolumn{1}{|r|}{$%
\begin{array}{c}
\text{{\small 28.62}} \\ 
\text{{\small [30.59]}}%
\end{array}%
$} \\ \hline
$1.8$ & $0.5$ & $\text{{\small 1.000}}$ & \multicolumn{1}{|r|}{$%
\begin{array}{c}
\text{{\small 1.018}} \\ 
\text{{\small (0.160)}}%
\end{array}%
$} & \multicolumn{1}{|r|}{$%
\begin{array}{c}
\text{{\small 18.35}} \\ 
\text{{\small [20.10]}}%
\end{array}%
$} & $\text{{\small 1.497}}$ & \multicolumn{1}{|r|}{$%
\begin{array}{c}
\text{{\small 1.512}} \\ 
\text{{\small (0.203)}}%
\end{array}%
$} & \multicolumn{1}{|r|}{$%
\begin{array}{c}
\text{{\small 29.80}} \\ 
\text{{\small [31.24]}}%
\end{array}%
$} \\ \hline
\end{tabular}%
\caption{Simulation results of the estimation of $\protect\delta$ 
from selected $\protect\textsc{sma}$(1) processes.
Reported in this table for each process are 
(i) the mean and standard deviation, in ( ), of $\protect\widehat{\protect\delta}$ 
across all realisations and 
(ii) the variance of $\protect\widehat{\protect\delta}$ across all realisations
multiplied by the sample size $n$ and the true asymptotic variance, in [ ].}%
\label{tab:sim:4}%
\end{table}%

\clearpage

\bibliographystyle{model2-names}
\bibliography{acompat,EstimationOfStableDistributionParametersFromADependentSample}

\end{document}